\documentclass[a4paper,12pt]{article}
\usepackage[utf8]{inputenc}
\usepackage[T1]{fontenc}
\usepackage[english]{babel}
\usepackage{amsmath}
\usepackage{amssymb}
\usepackage{amsthm}
\usepackage[all,cmtip]{xy}
\usepackage{graphicx}
\usepackage{xcolor}
\usepackage{array}
\usepackage{multirow}

\usepackage{enumerate}

\numberwithin{equation}{section}

\renewcommand{\thefigure}{\thefigure.\arabic{equation}}
\numberwithin{figure}{section}

\usepackage[normalem]{ulem}

\title{Laplacian Matrices}

\date{ }

\theoremstyle{plain}
\newtheorem{teo}{Theorem}[section]
\newtheorem{prop}[teo]{Proposition}
\newtheorem{lema}[teo]{Lemma}
\newtheorem{coro}[teo]{Corollary}

\newtheorem{defin}[teo]{Definition}
\newtheorem{expl}[teo]{Example}

\theoremstyle{definition}
\newtheorem{obs}[teo]{Remark}

\newcommand{\Rr}{\mathbb{R}}
\newcommand{\CG}{\mathcal{G}}
\newcommand{\CC}{\mathcal{C}}
\newcommand{\CE}{\mathcal{E}}

\newcommand{\CT}{\mathcal{T}}

\newcommand{\CM}{\mathcal{M}}

\newcommand{\tq}{\, | \,}

\begin{document}

\title{Synchrony patterns in Laplacian networks} 


\date{}
\maketitle

\vspace*{-1cm}
\centerline{\scshape Tiago de Albuquerque Amorim\footnote{Email address: tiagoamorim8@usp.br
 }}
{\footnotesize \centerline{Mathematics Department, ICMC}
\centerline{University of S\~ao Paulo} \centerline{13560-970 P.O. box
668, S\~ao Carlos, SP - Brazil }}
\vspace*{0,5cm}
\centerline{\scshape Miriam Manoel\footnote{Email address:
miriam@icmc.usp.br (corresponding author)}}
{\footnotesize \centerline{Mathematics Department, ICMC}
\centerline{University of S\~ao Paulo} \centerline{13560-970 P.O. box 668, S\~ao Carlos, SP - Brazil }}

\vspace*{1cm}

\begin{abstract}

A network of coupled dynamical systems is represented by a graph whose vertices represent individual cells and whose edges represent couplings between cells. Motivated by the impact of synchronization results of the Kuramoto networks,  we introduce the generalized class of Laplacian networks, governed by mappings whose Jacobian at any point is a symmetric matrix with row entries summing to zero. By recognizing this matrix with a weighted Laplacian  of the associated graph,  we derive the optimal estimates of its positive, null and negative eigenvalues directly from the graph topology. Furthermore, we provide a characterization of the mappings that define Laplacian networks. Lastly, we discuss stability of equilibria inside synchrony subspaces for two types of Laplacian network on a ring with some extra couplings. 

\vspace*{0,5cm}

\noindent {\bf Keywords.} network, admissible mapping, graph Laplacian, singularity, synchrony, stability. 

\noindent {\bf AMS classification.} 34C15, 37G40, 82B20, 34D06.
\end{abstract}

\section{Introduction} \label{sec:introduction}
About fifty years ago Kuramoto introduced in \cite{Kuramoto1975} a system of ordinary differential equations  that has since gained recognition as the simple (or traditional) Kuramoto model. It has been proposed as a straightforward and solvable framework to comprehend mutual synchronization within a cluster of oscillators that are equally coupled to all other oscillators. Its significance is further underscored by its applications across various synchronization scenarios in chemical \cite{CBV}, biological \cite{VHMP}, social systems  and in neuroscience phenomena \cite{NB}. 
Since then, numerous related articles have emerged in the literature in distinct contexts of investigation, varying from  numerical analysis and stochastic methods \cite{KJA} to criteria for the existence and stability of (clusters of) synchrony \cite{Bronski2}. In this paper we introduce a more comprehensive framework for this particular type of  system of coupled identical individual systems, accounting for two distinct aspects: 
\begin{itemize}
\item the coupling may not necessarily be universally applied across all elements;  
\item the Jacobian of the network vector field is a general Laplacian matrix.
\end{itemize}
As we shall see, the second requires that being bidirected is a necessary condition for the associated graphs. We remark that the second aspect   appears naturally in gradient systems of coupled cell networks with an extra ${\bf S}^1$-invariance condition on the vector field components; see \cite{MR}, where critical points are investigated for cells coupled in a ring. In fact, we prove that vector fields with a Laplacian structure are gradient; see Theorem~\ref{corograd}.  \\

Recall from algebraic graph theory that eigenvalues of a weighted graph Laplacian  matrix  hold substantial information of the graph. For instance, it follows directly from the definition that zero is an eigenvalue, associated with the eigenvector $(1,\ldots,1)$. This condition is already interesting from the point of view of the local dynamics, as it implies that there are no isolated singularities. Furthermore,
in the case of non-negative weights, it is well known that all the eigenvalues are non-negative and the algebraic multiplicity of the zero eigenvalue is exactly the number of connected components of the graph. This result has been extended in \cite{Bronski}  to the case of connected graphs allowing  negative weights. Since it is expected that the Jacobian matrix of the vector field has zero entries, recognizing this matrix as the graph Laplacian naturally leads us to the inherent condition that the graph may be disconnected. With the requirement to address this condition, we have interpreted the absence of an edge connecting two vertices as the existence of an edge with zero weight to extend the result  to disconnected graphs; see Theorem~\ref{Bronskiteo}. The result is the optimal estimate for the number of negative, null and positive eigenvalues from the number of connected components of the subgraph determined by positive weights and the subgraph determined by negative weights. \\











In the study of local dynamics, specially in bifurcation theory, the starting point is the analysis of existence and nature of equilibria and periodic orbits. An equilibrium is a critical point of the vector field expected to lye on subspaces that are invariant under the dynamics. In equivariant dynamics, these are fixed-point subspaces of subgroups of the symmetry group; in coupled networks, these are  synchrony subspaces determined from the associated graph. Synchrony subspaces are robust polysynchronous subspaces; that is, polysynchronous subspaces (determined by a set of equalities of the vector components -- the synchronized components) that are invariant under any vector field  admissible for the graph.  For a comprehensive study of the formalism of these two contexts,  we refer to \cite{GSS} and \cite{livro}, respectively. We now raise the issue regarding a connection between synchrony and symmetries of the automorphism group of the network graph. Not all  polysynchronous subspaces are robust; not all robust polysynchronous (synchrony) subspaces are fixed-point subspaces of symmetries. In fact, even for regular graphs, there are synchronies that do not come from symmetries of the graph automorphisms, and this has been named {\it exotic} synchrony by the authors in \cite{AntoneliStewart}. We prove that, for regular networks, all polysynchrony subspaces  are generically robust polysynchronous  (Proposition~\ref{prop:inv polydiagonal}). Furthermore,  we present two classes of regular graphs for which synchrony and symmetry coincide. One is the class of  regular rings (any number of cells)  and the other is the class of the so-called $G_n$-graphs ($5 \leq n \leq 9$ and $n = 11$); see Propositions~\ref{prop:ring} and  \ref{prop:Gm}. $G_n$ is the graph of $n$ cells connected by edges with nearest and next nearest neighbors. Subsection~\ref{subsec:G6} presents the study of singularities of a Kuramoto network with coupling determined by the graph $G_6$. \\

As the first step to extend the simple Kuramoto model, we characterize all vector fields whose Jacobian at any point is a Laplacian matrix (Theorem \ref{them:general form}). We also deduce that these belong to a subset of gradient fields, namely those whose potential function operates on the differences between variables. In particular, it follows that the dynamics is completely understood from the study of critical points, from the LaSalle's principle of invariance. The results on stability of critical points then follows from an investigation of signs of the Laplacian matrix eigenvalues that can be deduced from Theorem~\ref{Bronskiteo}. Yet based on the additive nature of Kuramoto systems, in Subsection~\ref{subseq:additive structure} we impose an additive structure on the couplings. Under this condition,  each component of the mapping being odd is a necessary condition (Theorem \ref{them:additive component}); for the Kuramoto network, this is the sine function.  As an advantage, this structure provides a direct way to associate an admissible graph to this mapping, in a unique way,  from the steps given in the formalism of  \cite{AM}. This, in turn, enables us to use graph topology tools to study synchronization. \\

The authors in \cite{JMB} show that for the simple Kuramoto network the total synchrony is an asymptotically stable family of critical points. Here we show that this also holds for the generalized  Laplacian additive networks (Theorem~\ref{maintheo2}). As mentioned above, the authors in \cite{MR} describe the critical points  on synchrony subspaces for gradient networks of identical cells on a ring with the additional condition of $S^1$-invariance in the coupling function. We notice that these turn out to be a particular case of Laplacian additive networks; for the details we refer to \cite[Section~4]{MR}. Here we present the complete list of critical points and their stability estimates for two examples of Laplacian additive network. As a first example, we consider a network with six identical cells with couplings determined by the regular network graph $G_6$, with the coupling function to be the sine function as in the traditional Kuramoto model (Subsection~\ref{subsec:G6}). In the second example we consider a homogeneous network with six identical cells and two types of couplings, namely a combination of the sine and the identity functions.    \\

Let us summarize how this paper is organized. In Section~\ref{sec:preliminaries} we briefly review the network formalism introduced in \cite{SMP}. In Section~\ref{sec:synch and symm}  we review the notions of balanced relation and robust polysynchrony, also introduced in \cite{SMP}, and we discuss about invariant subspaces of a regular graph network dynamics  regarding their connection with synchrony and symmetry.  Section~\ref{sec:Laplacian} is devoted to our two main results on Laplacian networks,  given by Theorems~\ref{Bronskiteo} and \ref{them:general form}, where we also discuss the case of additive structure on the couplings. Finally, in Section~\ref{sec:examples} we deduce Lyapunov stability for totally synchronous equilibria in additive Laplacian networks (Theorem~\ref{maintheo2}) and for two particular examples we study the existence and nature of critical points on the remaining synchrony subspaces. 

\section{Coupled cell networks} \label{sec:preliminaries}

In this section we briefly recall the set up of the network structure of a dynamical system defined by an autonomous system of ordinary differential equations
\begin{equation} \label{eq:ODE}
 \dot{x} \ = \ f(x), \ \ x \in M, 
\end{equation}
where $M$ is a smooth manifold, the {\it state space} or {\it phase space}, and $f$ is a smooth vector field. To simplify exposition, we shall mostly assume $M = {\mathbb R}^n$, so we  consider $f$ to be a map $f : {\mathbb R}^n \to {\mathbb R}^n$, and  write $f(x_1, \ldots, x_n) = \bigl( f_1(x_1, \ldots, x_n), \ldots, f_n(x_1, \ldots, x_n) \bigr),$ but we point out that  in Subsection~\ref{subsec:G6},  for example, $M$ is the 6-torus ${\mathbb T}^6$. 

The network structure appears when a component $f_c$, governing the dynamics of an individual cell $c$, depends on a subset of variables, 
 the {\it input} variables of cell $c$, and is independent of the remaining variables. In this setting, the phase space is a vector space of the form $ P = P_1 \times \cdots \times P_n$, where $x_c \in P_c$ may be multidimensional. Here we assume that the cells are identical, so $P_c$ is the same for all $c$.
We say that cell $c$ is interacting with cells associated with its input set of variables.  The network is then represented by a graph, with the set of vertices representing the set of cells  and the edges representing the interactions. 

More precisely, a {\it coupled cell network} $\CG$ consists of a finite set of cells $\CC=\{1,\ldots,n\}$ and a finite set of arrows (or edges) $\CE \subset \CC \times \CC$, both equipped with equivalence relations $\sim_{\CC}$ and $\sim_{\CE}$, 
 satisfying the compatibility condition
\begin{eqnarray} (c_1,d_1)\sim_{\CE}(c_2,d_2) \Rightarrow  c_1 \sim_{\CC} c_2,  \ d_1 \sim_{\CC} d_2. \nonumber 
\end{eqnarray}
For simplicity we assume that there is no arrow of the form $(c,c)$. Distinct interactions in a coupled cell network define distinct $\sim_{\CE}$-classes of edges, so to each class  $\xi \in \CE/{\sim_{\CE}}$ there corresponds an adjacency matrix $A^{\xi},$ 
\begin{eqnarray} \label{eq:adjacency matrix}
(A^{\xi})_{ij}= \left\{ \begin{array}{c} 1, \mbox{ if } (j,i) \in \xi \\ 0, \mbox{ if } (j,i) \not\in \xi \end{array} \right.  
\end{eqnarray}
If the model requires, we similarly associate a weighted graph to the network, in which case  
\begin{equation} \label{eq:wij}
(A^{\xi})_{ij} = w_{ij}, 
\end{equation}
where $w_{ij}$ denotes the weight of (the class of) the edge $(j,i)$. In the present paper, the interest lies on bidirected graphs; that is, $(i,j)$ is an edge if, and only if $(j,i)$ is an edge of the same type; see the remark after Definition~\ref{def:Laplacian network}. Hence, the adjacency matrices are symmetric. 

We now establish the algebraic set up of a coupled cell network. Let
$$I(c)=\{e \in \CE : \  e=(d,c)\}$$ 
denote the {\it input  set} of a cell $c$. 
Another equivalence relation among cells is given by an identification of input sets: $c$ and $d$ are said {\it input  equivalent} if there exists an arrow-type preserving bijection 
\begin{eqnarray} 
\beta \colon I(c) \rightarrow  I(d), \nonumber 
\end{eqnarray} 
that is, $\beta(e)\sim_{\CE} e,$ for all $e \in I(c).$ Let $B(c, d)$ denote the set of all bijections as above.\\

We end with the two main definitions for the forthcoming sections. 
\begin{defin} \label{def:homog regular}
A network graph $\CG$ is  {\it homogeneous} if all cells are $I$-equivalent, in which case $B(c,d)  \neq \emptyset,$  for all $c,d \in \CC$.  A network graph  is {\it regular} if it is homogeneous with one type of edge.  
\end{defin}

If $I(c) = \{ i_1, \ldots, i_{v} \},$  let us denote  $x_{{I}(c)} = (x_{i_1}, \ldots, x_{i_\nu})$. Then:
\begin{defin} \label{def:adm map}
	A map $f \colon P \to P$ is $\CG$-admissible if:
	\begin{itemize}
		\item [(a)] (Domain condition) For every $c \in \CC$, the component $f_c$  depends only on $x_c$ and on the variables associated with 
  ${I}(c)$. That is, there exists $\hat{f}_c \colon P_c \times P_{I(c)}\rightarrow P_c$ such that
\[		f_c(x)=\hat{f}_c(x_c, x_{I(c)}).\]	
		\item [(b)] (Pullback condition) For every input equivalent pair $c,d \in \CC$ and every $\beta \in 
		B(c,d)$,
\[		\hat{f}_d(x_d, x_{{I}(d)})=\hat{f}_c(x_d, x_{I_{\beta(c)}}). \]
	\end{itemize}
\end{defin}
\noindent We denote the vector space of admissible maps by ${\cal M}({\CG}, P)$, or just ${\cal M}({\CG})$ when $P$ is clear from the context. 

\section{Invariant subspaces in regular networks}  \label{sec:synch and symm}

Subspaces that remain unchanged under the dynamics hold significance. This section focuses on synchrony patterns, as they constitute subspaces of dynamics invariance in the network context. We initiate by briefly revisiting the fundamental concepts.\\

In broad terms, synchrony is a natural concept in network dynamics related to a partition of the cells into subsets (often called clusters) such that all cells in the same cluster are synchronous. If synchrony occurs in the network system (\ref{eq:ODE}), then two or more cells of a solution $x(t)$ behave identically, that is, if $c$ and $d$ are any two of these cells, then 
$$x_c(t)=x_d(t), \, \forall t \in D,$$
where $D$ is a domain in ${\mathbb R}$ for which $x(t)$ exists. The authors in \cite{SMT, SMP} characterize occurrence of synchrony in a network dynamics from the graph architecture. This is given as follows. 

Given a network graph $\CG$, for an edge $e \in {\cal E}$, let $\CT(e)$ denote the vertex given by the tail of  $e$. Let $\bowtie$ be an equivalence relation on $\CC$. Then 
\begin{itemize}
	\item  $\bowtie$ is said to be {\it balanced} if $c \bowtie d$ implies that there exists  $\beta \in B(c,d)$ such that $\CT(e) \bowtie \CT(\beta(e)), \, \forall e \in I(c)$; 
	\item $\bowtie$ is {\it robustly polysynchronous} if for any choice of total phase space $P$ we have $f(\Delta_{\bowtie}) \subseteq \Delta_{\bowtie}$, 
 for all $f \in \CM(\CG)$, where 
 $$\Delta_{\bowtie}=\{x \in P \tq x_c=x_d \Leftrightarrow c \bowtie d\}. $$
\end{itemize}
We notice that the first of the notions above is associated with the graph $\CG$, whereas the second is related to the vector space $\CM(\CG)$, but it turns out that both notions are equivalent, as established in \cite[Theorem 4.3]{SMT}. 
Moreover,  the authors prove that these are also equivalent to
\begin{eqnarray} A_{\CG}^{\xi}(\Delta_{\bowtie}) \subseteq \Delta_{\bowtie}, \, \, \, \forall \xi \in \CE/{\sim_{\CE}}, \ \ {\rm when} \  P = {\mathbb R}^n,  \label{eq: condition to balanced} 
\end{eqnarray} 
that is, synchrony subspaces of ${\mathbb R}^n$ are those left invariant under all adjacency matrices of $\CG$. 

\begin{defin} \label{def:synchrony subspace}
When $\bowtie$ is balanced, $\Delta_{\bowtie}$ is called a {\it synchrony subspace}.  
\end{defin}
The practical way to express the resulting $\bowtie$-classes of a synchrony pattern is coloring vertices: cells in the same class have the same color and `receive from edges' vertices of the same color. 

\begin{obs} \label{rmk: code}{
Based on (\ref{eq: condition to balanced}), the authors in \cite{AguiarDiasAlg} develop, for a given network graph, an algorithm whose output is the set of synchrony subspaces. This is done through polydiagonal invariant subspaces from the eigenvectors of the Jordan decomposition of the adjacency matrices. As they show, it suffices to implement the algorithm for the case of regular graph, because all the  lattice of synchrony subspaces can be associated with a regular network through the notion they introduce of {\it minimal synchrony subspaces that are sum-irreducible}. Following this approach, we have recently implemented this algorithm with a code using the software {\it Mathematica}.  The code has been useful to compute the synchrony patterns presented in \cite{AM} as well as all the synchrony patterns in the present work. }
\end{obs}

\subsection{Invariant subspaces from synchronies}

In this subsection we relate synchrony subspaces of a network graph $\CG$ to invariant subspaces under the dynamics of (\ref{eq:ODE}) for  $f \in {\cal M}(\CG)$.  As we have  pointed out in \cite{AM}, even for regular graphs a coupled dynamical system may admit invariant polydiagonal subspaces that are not patterns of synchrony. However, the next result shows that generically those are in fact synchrony patterns.

\begin{prop}  \label{prop:inv polydiagonal}
For any regular network graph $\CG$ on any phase space $P$, polydiagonal subspaces of $P$ that are invariant under an admissible map are generically synchrony subspaces. 
\end{prop}
\begin{proof}
Based on (\ref{eq: condition to balanced}) we assume, without loss of generality, that each cell phase space is one dimensional, so $P = {\mathbb R}^n$, where $n$ is the number of cells. Since $\CG$ is regular, the Jacobian  of an admissible map $f$ at any point of total synchrony $\nu=(t,\ldots,t) \in \Rr^n$ is
\begin{eqnarray}
Jf(\nu)= \alpha (\nu) I +\beta (\nu) A,
\label{eq:inv polydiagonal} 
\end{eqnarray} 
where $A$ is the adjacency matrix of $\CG$. The result now follows observing that $\nu \in \Theta$ for any polydiagonal $\Theta$ and taking the generic condition as $b(\nu) \neq 0$ for some $\nu$. 
\end{proof}
In Subsection~\ref{subsec:G6} we investigate existence and nature of singularities inside synchrony subspaces.  For the ODE (\ref{eq: kuramoto G6}), it is straightforward that $b(\nu) = 1$, for any $\nu \in {\mathbb T}^n$, so the singularities we find are in fact all the possible equilibria inside invariant subspaces, by Proposition~\ref{prop:inv polydiagonal}.

\subsection{Synchrony and symmetry in regular networks}

Existence of symmetries and synchronies force existence of invariant subspaces, because of their fixed-point sets and robustly polysynchronous subspaces, respectively. These are related notions and we now turn to the issue of synchronies with respect to symmetries of a network graph.

A symmetry of $\CG$, or an automorphism of $\CG$, is a permutation $\gamma_{\CC} \colon \CC \rightarrow \CC$ such that 
$$	(c,d) \sim_{\CE} (\gamma_{\CC}(c),\gamma_{\CC}(d)),$$
which also induces a permutation on $\CE$. The group ${\rm Aut}(\CG)$ of all such automorphisms  is the symmetry group of $\CG$.

It is well known that symmetries of a network graph determine synchronies, for every subgroup $\Sigma < {\rm Aut}(\CG)$ defines a balanced equivalence relation on $\CG$ (\cite[Proposition~3.3]{AntoneliStewart}): $$c \bowtie_{\Sigma} d \Leftrightarrow \exists \gamma \in \Sigma : \gamma(c)=d.$$ 
This is the same as saying that the synchrony subspace $\Delta_{\bowtie_{\Sigma}}$ is the fixed-point subspace of $\Sigma$,
$$ {\rm Fix}(\Sigma) = \{ x \in P \ : \ \sigma x = x, \ \forall \sigma \in \Sigma\}.$$ 
As it is also well known, the converse does not hold in general;  that is, not all synchrony subspaces are fixed-point subspaces. Such synchrony has been named  an {\it exotic} synchrony after \cite{AntoneliStewart}. 
a balanced relation $\bowtie$ is called \textit{exotic} if there is no subgroup $\Sigma < {\rm Aut}(\CG)$ for which $\bowtie=\bowtie_{\Sigma}$. 

Exotic patterns can appear in homogeneous graphs and even in the most simple cases of regular graphs. The first example in the literature was presented in \cite{SMM}, the regular graph $G_{12}$ of twelve vertices with nearest and next nearest neighbors. 
Exotic synchronies also appear in  \cite{AM} in the classification of  admissible homogeneous graphs of a class of vector fields on ${\mathbb R}^6$  (see  \cite[Table 2]{AM}). Figure~\ref{fig:exotic} illustrates an exotic pattern of the graph $\CG$ which is number 6  in \cite[Table 2]{AM}). Disregard the colors in Fig.~\ref{fig:exotic} for the picture of $\CG$. We have that   
 ${\rm Aut}(\CG)$ is the group ${\mathbb Z}_2 = \  <(12)(36)(45)>$ and the synchrony pattern is $\{x_1 = x_4, x_2 = x_5, x_3 = x_6\}$, which is clearly exotic.  

\begin{figure}[h] 
\begin{center}
   {\includegraphics[width=4.2cm]{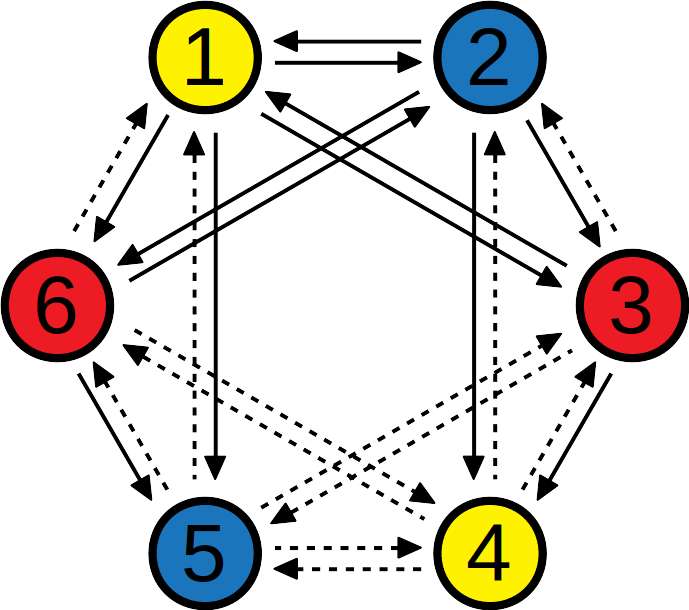}} 
  \end{center}
\caption{A homogeneous graph with six identical cells and two types of edges. The coloring shows an exotic synchrony.} \label{fig:exotic}
\end{figure}

There are however classes of network graphs for which all synchronies are determined by the symmetries  of ${\rm Aut}(\CG)$. For example, the regular network graph ${\rm G}_6$ of six cells with nearest and next nearest neighbor coupling of Subsection~\ref{subsec:G6}, which falls in the set of graphs given in  Proposition~{\ref{prop:Gm}}. The next proposition states that this is also the case for the regular ring  of $n$ cells with nearest neighbor coupling, for which ${\rm Aut}(\CG) = {\mathbb D}_n$. Although the result is not surprising, the proof is not straightforward and, to our knowledge, it has not been found in the literature, so it is included here. 
\begin{prop} \label{prop:ring}
There are no exotic pattern of synchrony for the  regular ring  of $n$ cells,  with nearest neighbor coupling, for $n \geq 3$. 
\end{prop}
\begin{proof} 
If $\bowtie$ is a balanced equivalence relation, we prove that  $\bowtie \ = \ \bowtie_{\Sigma}$ for some subgroup $\Sigma <  \mathbb{D}_n$. 

The following is a necessary condition for the equivalence relation $\bowtie$ to be balanced for this network graph:
\begin{eqnarray}c \bowtie d \quad  \Rightarrow \left\{ \begin{array}{c} (c-1 \bowtie d-1, \, \, c+1 \bowtie d+1) \\ \mbox{ or } \\  (c-1 \bowtie d+1, \, \,  c+1 \bowtie d-1). \end{array} \right.\label{eq: balanced on ring} \end{eqnarray}
First, suppose  that there are three consecutive  $\bowtie$-equivalent cells. By \eqref{eq: balanced on ring}, all cells must be $\bowtie$-equivalent, and so $\bowtie \ = \ \bowtie_{\mathbb{D}_n}$. 

We now suppose $c \bowtie d $ such that 
\begin{eqnarray}
c-1 \bowtie d-1, \, \, c+1 \bowtie d+1. \label{eq: bowtie on ring} \end{eqnarray} 
We need to identify $\Delta_{\bowtie}$ as a fixed-point subspace of  a subgroup  $\Sigma < {\mathbb{D}_n}$, that is,  find a rotation multiple of ${2\pi}/{n}$ or a line reflection generating $\Sigma$.

We can assume that $d-c>0$ is minimal with respect to \eqref{eq: bowtie on ring}. There is a sequence  of $\bowtie$-classes with length $d-c$ that repeats inside the ring. In fact, by \eqref{eq: balanced on ring}, 
$$ c\bowtie d, \  c+1\bowtie d+1 \Rightarrow c+2\bowtie d+2 \Rightarrow \ldots \Rightarrow c +(d-c)=d\bowtie d+(d-c). $$ 
Hence, $d-c$ divides $n$ and then a rotation of  ${2(d-c)\pi}/{n}$  is a generator of $\Sigma$. 

Finally, suppose that there are two cells $c' \bowtie d'$ such that $c<c'<d'<d$ and by minimality of $d-c>0$ we have that $c'-1 \bowtie d'+1, \, \,  c'+1 \bowtie d'-1$. The line reflection that takes $c'$ into $d'$ also takes $c'-1$ into $d'+1$, $c'+1$ into $d'-1$ and so on. Then, this reflection is a generator of $\Sigma$. 
\end{proof}

\begin{obs}
{\rm  
A study of a gradient network of identical oscillators on a ring is carried out in \cite{MR}. 
We notice that in the description given in  \cite[Table~1]{MR} a symmetry in one type of critical points in that table is missing. In fact, inside their list of critical points with trivial isotropy, there is a subclass with nontrivial isotropy, namely  $ {\mathbb Z}_2 < {\mathbb D}_n$, which is supported by Proposition~\ref{prop:ring} above.
Nonetheless, we point out that this does not affect the list of all possible critical points presented in that table, which is in fact complete.    
}
\end{obs}

The following result gives another set of regular network graphs for which all synchronies are symmetries:

\begin{prop} \label{prop:Gm}
There are no exotic pattern of synchrony for the  network graph $G_n$, for $5 \leq n \leq 9$ and $n=11$. 
\end{prop} 
\begin{proof}
This follows by an extensive use of the  {\it Mathematica} code; see Remark~\ref{rmk: code}.
\end{proof}
\noindent The list of exotic patterns of $G_{10}$ is given in \cite[Section 7]{AntoneliStewart}, which we have confirmed with our code. The network graph $G_{12}$ presents an exotic pattern, as already mentioned above. Due to the computational complexity, we have not investigated existence of exotic patterns in $G_n$ for $n \geq 13$ yet.  Some attempts suggest that there may be a general procedure to answer the question for bigger $n$ and we intend to go in this direction soon.

\section{Laplacian networks} \label{sec:Laplacian}
This section is devoted to the main object of this paper, namely the class of networks whose admissible maps are attached to the associated graph not only through its architecture but also through the Laplacian matrix of the graph. As mentioned in the introduction, these generalize the traditional Kuramoto model. 

In Subsection~\ref{subseq:Laplacian matrix} we recall the graph Laplacian matrix and present a result from algebraic graph theory for Laplacian matrices. Subsection~\ref{subsec:Laplacian maps} is dedicated to present the general algebraic expression of an admissible map of a Laplacian network. Finally, the extra admissible additive structure behind the Kuramoto networks is dealt with in Subsection~\ref{subseq:additive structure}. 

\subsection{Graph Laplacian matrix} \label{subseq:Laplacian matrix}
For a bidirected graph $G$ with $n$ vertices, recall that the Laplacian matrix is
\[ L = D - A,\]
where $D$ is the diagonal matrix of the valencies of the vertices and  $A$ is the  adjacency matrix.  The $ii$-entry of $D$ is the degree of the vertex $i$ if $G$ is unweighted,  and ${\sum_{j=1}^n w_{ij}}$ if $G$ is a weighted graph. $A$ is, as usual, the sum  of the adjacency matrices (\ref{eq:adjacency matrix})  (for unweighted edge classes) or (\ref{eq:wij}) (for weighted edge classes). \\

It is direct from the definition that $0$ is an eigenvalue of $L$ with eigenvector $(1, \ldots, 1)$. For weighted connected graphs with possibly negative weights,  the authors in \cite[Theorem 2.10]{Bronski} give the best possible bounds on the numbers of positive, negative and zero eigenvalues. The next result generalizes their result for disconnected graph,  which is used in the next sections. 

Let $G_+$ (resp. $G_-$) denote  the subgraph  with the same vertex set as $G$ together with the edges of positive (resp. negative) weights. Let  $n_0, n_-$ and $n_+$ denote the numbers of zero, negative and positive eigenvalues of the Laplacian $L$, and let  $c(G)$ denote the number of connected components of  $G$. 

\begin{teo} \label{Bronskiteo}
If $G$ is a (possibly disconnected) graph, then  
\begin{eqnarray}
c(G_+) -c(G) \, \,  \leq & n_+(L) & \leq \, \, n - c(G_-) \nonumber \\
c(G_-) -c(G) \, \,  \leq & n_-(L) & \leq \, \, n - c(G_+) \label{eq:estimates} \\
c(G) \, \, \leq & n_0(L) & \leq \, \, n+2c(G)-c(G_+)-c(G_-) . \nonumber 
\end{eqnarray}
\end{teo}

\begin{proof}
If $G$ has $k$ connected components, for $i = 1, \ldots k$, let $G_i$ denote the connected components with $n_i$ vertices, so $ n_1 + \ldots +n_k=n$. Reorder the rows and columns of $L$ if necessary to assume that  $L$ is a block diagonal matrix $$L= diag(L_1, \ldots, L_k).$$ 
By the estimates in \cite{Bronski}, for each $i=1, \ldots,k$ we have $c(\CG_{i+}) -1 \, \,  \leq n_+(L_i) \leq \, \, n_i - c(G_{i-})$, so
\begin{eqnarray} 
c(G_{1+})+ \ldots +c(G_{k+}) -k \, \,  &\leq& n_+(L_1)+ \ldots + n_+(L_k) \leq \\ \nonumber &\leq& \, \, n_1+ \ldots +n_k - (c(G_{1-}) + \ldots + c(G_{k-})), \nonumber
\end{eqnarray}
so the first estimate follows. The other inequalities follow similarly. 
\end{proof}
\noindent If the graph $G$ is connected, then $c(G) = 1$ and (\ref{eq:estimates}) are the bounds given in \cite{Bronski}. \\

We notice that this result involves only topological information about the graph, namely the connectivity of the graph and the sign information on the edge weights. In particular,  it follows that, for a graph with $n$ vertices, the difference between the upper and lower bounds in (\ref{eq:estimates}) is an integer that can vary between 0 and $n-1$.

\subsection{Laplacian mappings} \label{subsec:Laplacian maps}

In this subsection we use the idea of inserting weights on an unweighted graph $G$ 
to adapt the algebraic graph theory for weighted graphs to the associated Laplacian network graph $\CG$. Let us explain: for a given unweighted graph $G$, we  associate an admissible coupled system (\ref{eq:ODE}). For our case of interest, the Jacobian  $Jf(x)$ of the governing vector field is, at any point $x$, a Laplacian matrix. The idea is to look at it as a weighted Laplacian of $G$, considering $G$ as a weighted graph. Under this approach, we use Theorem~\ref{Bronskiteo}. \\

We start with three definitions:

\begin{defin} \label{def:Laplacian matrix}
A matrix  $L =(l_{ij})$ of order $n$ is a Laplacian matrix if it is symmetric with $l_{ii}=-\sum_{j \neq i}l_{ij}$, for $i = 1, \ldots, n$.  
\end{defin}

Our interest here is to deduce an algebraic expression of a map 
$f : P \to P$ on a real vector space $P$,  so for simplicity we assume $P = {\mathbb R}^n$. 

\begin{defin} \label{def:Laplacian map}
A map $f: {\mathbb R^n} \to {\mathbb R}^n$ of class $C^1$ is a Laplacian map if its Jacobian $Jf(x)$ at any $x \in \Rr^n$ is a linear map whose matrix is a Laplacian matrix.
\end{defin}
 \noindent The set of Laplacian maps shall be denoted by $LS(\Rr^n)$. \\

It follows from the two definitions above that being bidirected is a necessary assumption in this context. We have:
 
\begin{defin} \label{def:Laplacian network}
For a bidirected graph, an associated network $\CG$ is a Laplacian network if the admissible map is a Laplacian map.
\end{defin}

For the main result, we need:

\begin{lema} \label{lem:zero sum}
If $h: {\Rr}^n \to {\Rr}$ is such that
\begin{eqnarray} \label{eq1}
\sum _{i =1}^n \frac{\partial h}{\partial x_i}(x) =0.
\end{eqnarray}
then  
\begin{equation} \label{eq:component}
h(x_1, \ldots, x_n)=\alpha(x_1-x_n, \ldots, x_{n-1}-x_n),
\end{equation}
for some $\alpha : \Rr^{n-1} \to \Rr$.
\end{lema}
\begin{proof}
Consider the change of coordinates $t_i = x_i - x_n,$ $i = 1 \ldots, n-1$, $t_n = x_n$.
We then have
\[ h(x_1, \ldots, x_n) = h(t_1+t_n, \ldots, t_{n-1} + t_n, t_n) =  \tilde{\alpha}(t_1, \ldots, t_n).  \]
But
\[ \frac{\partial \tilde{\alpha}}{\partial t_n}=\sum_i \frac{\partial h}{\partial x_i}=0.    \]
Hence,
$\tilde{\alpha}(t_1, \ldots, t_n) = \alpha(t_1, \ldots, t_{n-1}),$ 
and the result follows.
\end{proof}

We now present the characterization of the general form of a Laplacian mapping:

\begin{teo} \label{them:general form}
$ f=(f_1, \ldots, f_n) \in LS(\Rr^n)$ if, and only if,    
\begin{eqnarray}
f_i(x_1, \ldots, x_n)&=&\dfrac{\partial g}{\partial t_i}(x_1-x_n, \ldots, x_{n-1}-x_n), \quad i=1, \ldots, n-1, \nonumber \\
f_n&=&-f_1-\ldots -f_{n-1} + k. \nonumber 
\end{eqnarray} 
for some  $g  \in C^2(\Rr^{n-1})$ and some constant $k \in \Rr$.
\end{teo}
\begin{proof}
From Lemma~\ref{lem:zero sum}, we have
\begin{eqnarray}
f_i(x_1, \ldots, x_n)&=&\alpha_i(x_1-x_n, \ldots, x_{n-1}-x_n), \quad i=1, \ldots, n. \nonumber 
\end{eqnarray} 
The Jacobian $Jf(x)$ is symmetric, so
\begin{eqnarray}
\frac{\partial \alpha_i}{\partial t_j}=\frac{\partial \alpha_j}{\partial t_i},  \quad i,j=1, \ldots, n-1, i \neq j, \label{eq6} \\ 
\frac{\partial \alpha_n}{\partial t_i} = - \frac{\partial \alpha_i}{\partial t_1}- \ldots - \frac{\partial \alpha_i}{\partial t_{n-1}}  \quad i=1, \ldots, n-1,  \label{eq7}
\end{eqnarray}  
which imply that
\begin{eqnarray} \label{eq10}
\frac{\partial \alpha_n}{\partial t_i} = - \frac{\partial \alpha_1}{\partial t_i}- \ldots - \frac{\partial \alpha_{n-1}}{\partial t_{i}}  = \frac{\partial }{\partial t_i}(-\alpha_1- \ldots - \alpha_{n-1}), 
\quad i=1, \ldots, n-1.
\end{eqnarray}
Hence, $\alpha_n = - \sum_{i=1}^{n-1} \alpha_i + k$, for some constant $k.$ 

Finally, given any function  $g \colon \Rr^{n-1} \rightarrow \Rr$ of class $C^2$, we can take $\alpha_i=\dfrac{\partial g}{\partial t_i}$, $i=1, \ldots, n-1$. Conversely, any map $\alpha=(\alpha_1, \ldots, \alpha_{n-1})$ of class $C^1$  satisfying  \eqref{eq6} is a gradient mapping. In fact, just set 
$$g(t_1, \ldots,t_{n-1})=\sum_{i=1}^{n-1}\int_0^{t_i}\alpha_i(0, \ldots,0,s,t_{i+1}, \ldots, t_{n-1})ds.$$ 
\end{proof}
The example below is a simple illustration of Theorem~\ref{them:general form}.
\begin{expl}
	For $n=3$, consider $g(y,z)=\dfrac{y^2z^2}{2}$. The following are the coordinate  functions of a Laplacian mapping $f \colon \Rr^3 \rightarrow \Rr^3$:
	\begin{eqnarray}
	f_1(x_1,x_2,x_3)&=&(x_1-x_3)(x_2-x_3)^2 \nonumber \\
	f_2(x_1,x_2,x_3)&=&(x_1-x_3)^2(x_2-x_3) \nonumber \\
	f_3(x_1,x_2,x_3)&=&-(x_1-x_3)(x_2-x_3)^2-(x_1-x_3)^2(x_2-x_3). \nonumber 
	\end{eqnarray}
\end{expl}

The next two corollaries follow straightforwardly:

\begin{coro} 
The following linear isomorphism  holds:
$$LS(\Rr^n) \ \simeq (C^2(\Rr^{n-1})/\Rr) \oplus \Rr.$$ 
\end{coro}

\begin{coro} \label{corosumzero} 
If $f=(f_1, \ldots, f_n) \in LS(\Rr^n)$, then $\sum_{i=1}^n f_i$ is constant. 
\end{coro}

We end with the characterization of the Laplacian maps, which is also direct from Theorem~\ref{them:general form}:

\begin{teo} \label{corograd}
$f=(f_1, \ldots, f_n) \in LS(\Rr^n)$ if, and only if,  $f$ is a gradient map $f= \nabla \bar{g}$, where $\bar{g}(x_1, \ldots, x_n) =g(x_1-x_n, \ldots, x_{n-1}-x_n) + kx_n$. 	
\end{teo}

\subsection{Laplacian networks with additive structure} \label{subseq:additive structure}

Here we consider Laplacian admissible maps $f \in LS(\Rr^n)$ with the extra condition that each component $f_c$ of cell $c$ has an additive input structure, namely 
\begin{equation} \label{eq:additive}
f_c(x)=\sum_{d\neq c} \phi_{cd}(x_c,x_d).
\end{equation}
It is worth mentioning that this structure provides the simplest way to associate an admissible graph following the formalism in \cite{AM}:  for such a map,  we consider the following coupling rule:  
\begin{eqnarray} \label{eq:optimized}
(a,b) \sim_{\CE} (c,d) \ \Leftrightarrow \  \phi_{ba}=\phi_{dc}. 
\end{eqnarray}
The question of realization of admissible graphs for a given map has been answered in \cite{AM}, and the notion of the optimized admissible graph for this map is given. In this setting, we point out that, in the present context, the condition (\ref{eq:optimized}) realizes a unique optimized admissible  graph  for the map as in (\ref{eq:additive});  see \cite[Remark 3.1 b]{AM}. 

We notice that, by Lemma~\ref{lem:zero sum}, 
\begin{eqnarray} f_n(x_1, \ldots, x_n)=\alpha_n(x_1-x_n, \ldots, x_{n-1}-x_n)
\end{eqnarray}
and from (\ref{eq:additive}),   
\begin{eqnarray} \frac{\partial^2 \alpha_n}{\partial t_i \partial t_j}=\frac{\partial^2 f_n}{\partial x_i \partial x_j}=0, \quad i \neq j <n. \label{eq11}
\end{eqnarray} 
Hence, $\alpha_n$ `splits the variables' $x_i$ and $x_j$, $i \neq j < n$; more precisely, 
$$f_n(x_1, \ldots, x_n)=\sum_{d < n} \phi_{nd}(x_n -x_d).$$
By the same reasoning, the same condition holds for other components, and so 
$$f_c(x)=\sum_{d \neq c} \phi_{cd}(x_c-x_d), \quad \forall c=1, \ldots, n. $$

Since the graph is bidirected, we have  $(a,b) \sim_\CE (b,a)$, and so  $\phi_{ba}=\phi_{ab}$. Furthermore, by the Laplacian condition,  $$\phi_{ab}'(x_a-x_b) = - \dfrac{\partial f_a}{\partial x_b}(x) =  - \dfrac{\partial f_b}{\partial x_a}(x) = \phi_{ba}'(x_b-x_a).$$
Therefore, 
$$\phi_{ab}'(x_a-x_b) =  \phi_{ab}'(-(x_a-x_b)), $$
that is, all coupling functions must be odd.  We have just proved:

\begin{teo}\label{teo2} \label{them:additive component}
Let $\CG$ be a  bidirected graph network. If $f$ is a $\CG$-admissible map, then $f \in LS(\Rr^n)$ has an additive structure if, and only if, each component of $f$ is of the form 
\begin{eqnarray} f_c(x)=k_{[c]}+ \sum_{d \in  I(c)} \phi_{[(c,d)]}(x_d-x_c), \label{eq3}
\end{eqnarray}
where $\phi_{[(c,d)]}$ is an odd function that depends only on the $[(c,d)]$-class of the edge $(c,d)$ and $k_{[c]}$ is a constant that depends only on the $I$-class $[c]$ of $c$. 
\end{teo}

The two theorems  above imply the following: 

\begin{coro} \label{coroadd}
Let $\CG$ be a bidirected graph network. Let $f$ be a $\CG$-admissible map. Then $f \in LS(\Rr^n)$ has an additive structure if, and only if, $f$ is a gradient mapping $f= - \nabla {g}$, where 
\[ {g}(x)=\sum_{  (c,d) \in \CE} \psi_{[(c,d)]}(x_d-x_c) + \sum_{c \in \CC}k_{[c]}x_c, \] 
$\psi_{[(c,d)]}$ is an even function that depends only on the  $[(c,d)]$-class of the edge $(c,d)$ and $k_{[c]}$ is a constant determined by the $I$-class of $c$.

\end{coro}

The example  below illustrates the theorem above.  This is a nonhomogeneous network graph with six cells, with two $I$-classes $\{1, 2, 4, 5\}$ and $\{3, 6\}$, given in Fig.~\ref{fig:example}. 

\begin{figure}[ht] 
	\centering 
	\includegraphics[width=4.2cm]{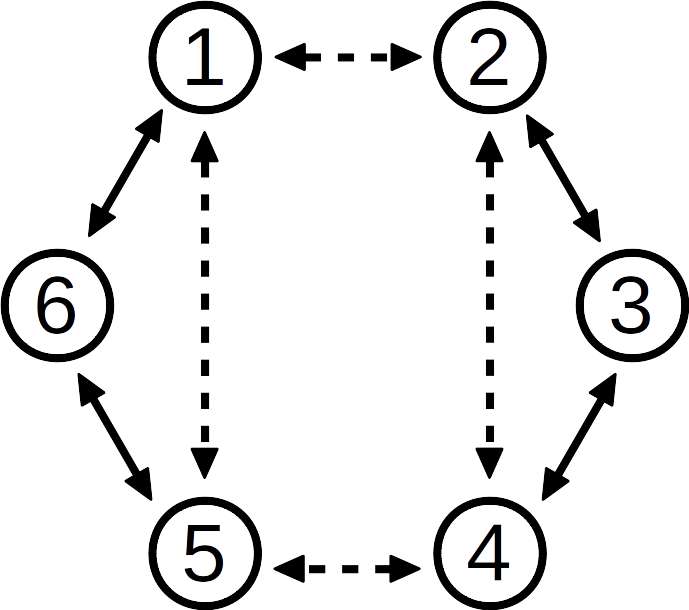} 
	\caption{ A nonhomogeneous bidirected graph with six identical cells and two types of edges.} \label{fig:example}
\end{figure}

\begin{expl}
For a bidirected graph with six identical cells and two types of edges as given in Fig.~\ref{fig:example}, a Laplacian network $\CG$ on ${\Rr}^6$ with additive structure is given by

\[ \dot{x} = f(x), \]
where $f = (f_1, \ldots, f_6) \colon \Rr^6 \rightarrow \Rr^6,$ 
\begin{eqnarray}
f_1(x)&=& \kappa+\theta(x_1-x_2)+\theta(x_1-x_5)+\phi(x_1-x_6) \nonumber \\
f_2(x)&=& \kappa+\theta(x_2-x_1)+\theta(x_2-x_4)+\phi(x_2-x_3) \nonumber \\
f_3(x)&=& \ell+\phi(x_3-x_2)+\phi(x_3-x_4) \nonumber \\
f_4(x)&=& \kappa+\theta(x_4-x_2)+\theta(x_4-x_3)+\phi(x_4-x_5) \nonumber \\
f_5(x)&=& \kappa+\theta(x_5-x_1)+\theta(x_5-x_4)+\phi(x_5-x_6) \nonumber \\
f_6(x)&=& \ell +\phi(x_6-x_1)+\phi(x_6-x_5), \nonumber
\end{eqnarray}
for $\theta, \phi$  any odd functions of class $C^1$ and $\kappa, \ell$ constant.
\end{expl}

\section{Critical points on synchrony subspaces of additive Laplacian networks} \label{sec:examples}

The starting point in the study of a dynamical behaviour of a system, or bifurcations with variations of possible external parameters,  is the analysis of existence and stability of equilibrium points. Here we proceed in this direction. \\

In Subsection~\ref{subsec:total synchronies} we prove that, for any homogeneous additive Laplacian network, Lyapunov stability holds generically for totally synchronyous critical points. In the last two subsections we choose two examples to search for critical points with the remaining possible synchrony.  

The following result ensures that the investigation of trajectory stability in Laplacian networks is based on the eigenvalues of the Jacobian evaluated at equilibrium points.

\begin{lema} 
Let $f \in LS(\Rr^n)$.	 The $\alpha$-limit and $\omega$-limit of any trajectory are equilibrium points. 
\end{lema}	

\begin{proof}	
From Theorem~\ref{corograd}, $f \in LS(\Rr^n)$ is gradient, so the result follows from LaSalle's invariance principle (see for example \cite[Section 8.3]{Wiggins}). 
\end{proof}

In Subsection~\ref{subsec:G6} we consider the homogeneous Kuramoto network of six cells with coupling  $G_6$ (identical edges connecting nearest and next nearest neighbors). We list the critical points inside the remaining synchrony subspaces. For each case, we give the estimate for the degree of linear instability through the range of the number of  positive eigenvalues. We end with Subsection~\ref{subsec:two types of edges} doing the same 
analysis for a homogeneous coupled network of six cells with two types of couplings.

\subsection{Total synchrony subspace} \label{subsec:total synchronies}

Let $\CG$ be a connected homogeneous graph
with $n$ cells. 
Let $f$ be an admissible  vector field as in \eqref{eq3} . We search for  critical points, so $k=0$. \\

Suppose that there exists $\varepsilon>0$ such that \begin{eqnarray} \alpha\phi_{[(c,d)]}(\alpha)>0, \forall  \alpha \in (-\varepsilon,\varepsilon),\alpha\neq  0. \label{condition on phi} \end{eqnarray}
Notice that this holds if, in particular, $\phi'_{[(c,d)]}(0)>0$, for all  $(c,d) \in \CE$, which can be the case if the primitive of $\phi_{[(c,d)]}$ has a local minimum at zero. \\

Let $\Omega$ be the biggest open hipercilinder around of the total synchrony subspace $\Delta$ such that $x \in \Omega$ implies that $|x_d-x_c|< \varepsilon$, for all  $(c,d) \in \CE$.  We then have the following:   
 
 \begin{lema} \label{lemaOmega}
\begin{itemize} \item [(i)] $\Omega$ is a flow invariant set.
	 \item [(ii)] for $x \in \Omega$,  we have $f(x)=0$ if, and only if,  $x \in \Delta$.
	 \end{itemize}
 \end{lema}

Before proving this lemma, we state the following theorem, which is an immediate consequence of this Lemma, since $f$ is gradient (see Corollary \ref{coroadd}).

\begin{teo} \label{maintheo2}
Let $\CG$ be a connected homogeneous graph.
Let $f$ be an admissible vector field as in \eqref{eq3} such that \eqref{condition on phi} holds. The total synchrony subspace is asymptotically stable on $\Omega$ in the sense of Lyapunov. 
\end{teo}

\begin{proof}[ Proof of Lemma \ref{lemaOmega}]
(i) From Corollary~\ref{corosumzero}, $f(x) \cdot  (1, \ldots,1)  =0$ for any $x$. Also, $f$ is constant along the lines parallel to $\Delta$. So it suffices to show that $\Omega \cap \pi$ is flow invariant, where $\pi$ is the hiperplane $ x \cdot (1, \ldots,1)  =0$. To see this, as $\Omega \cap \pi$ is a hipersphere, we just need to show that $f(x) \cdot x < 0$  for all  $x \in \Omega \cap \pi \backslash \{0\}$. In fact, 
\begin{eqnarray}
f(x) \cdot x &=& \sum_{c \in \CC}  \sum_{  d \in I(c)} x_c\phi_{[(c,d)]}(x_d-x_c) \nonumber \\ & =& \sum_{(c,d) \in \CE} (x_c-x_d)\phi_{[(c,d)]}(x_d-x_c)\leq 0, \label{eq06} 
\end{eqnarray}
where at least one portion is strictly less than zero. 
As $\CG$ is connect, \eqref{eq06} also shows (ii). \end{proof}

We finally mention that if all the functions $\phi_{[(c,d)]}$ are periodic with a comum period, we can restrict the analysis of critical points on the $n$-torus.

\subsection{Kuramoto network with  $G_6$ coupling} \label{subsec:G6}

We consider the Kuramoto network system with six cells and coupling 
 $G_6$ (Fig.~\ref{fig1}):
\begin{eqnarray}
\dot{x}_i=\sum_{j=-2}^2 \sin({x_{i-j}-x_i}), \qquad i=1, \ldots,6. \label{eq: kuramoto G6}
\end{eqnarray}

\begin{figure}[h]
	\centering 
	\includegraphics[width=4.2cm]{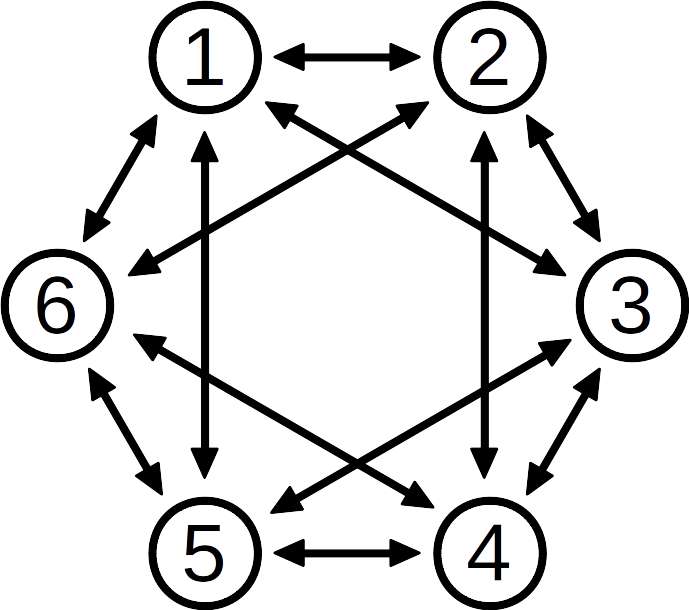} 
	
	\caption{$G_6$} \label{fig1}
\end{figure}

We search for the critical points inside synchrony subspaces. It follows from Proposition~\ref{prop:Gm}  that these  are all fixed-point subspaces of subgroups of the automorphism group $Aut(G_6)$, which is the octahedral group. We consider the representation given by permutations,   
$$Aut(G_6)=\langle \mathbb{D}_6,(14),(25)\rangle $$ 
(see \cite[Lemma 2.1]{AntoneliStewart}).

The computation has been carried out starting with the smallest subgroups $\Sigma$ of $Aut(G_6)$ up to conjugacy, for these will have the largest fixed-point subspaces; and we then continue refining from there. More precisely, we consider the diagram (\ref{eq:diagram}) with the lattice of fixed-point subspaces, where the arrows represent the refinement by inclusion. The total synchrony has been omitted in this diagram and, apart from it, this is the complete lattice up to conjugacy. 
Based on this diagram, we start by considering the subgroups $\Sigma$ such that Fix($\Sigma$) correspond to patterns 1 and 3.

\begin{eqnarray}  \label{eq:diagram}
\xymatrix{ \overset{Pattern \, 8}{{\rm Fix}\langle (14)(25)(36) \rangle} \ar[r] &\overset{Pattern \,4}{{\rm Fix}\langle (14)(25) \rangle} \ar[r] & \overset{Pattern \, 1}{{\rm Fix}\langle (14) \rangle}   \\ \overset{Pattern \,5}{{\rm Fix}\langle (1245)(36) \rangle} \ar[r] \ar[u] \ar[d] &\overset{Pattern \,2}{{\rm Fix}\langle (1245) \rangle} \ar[ru] \ar[u] \ar[d] & \\ \overset{Pattern \,7}{{\rm Fix}\langle (12)(36)(45) \rangle}\ar[r] & \overset{Pattern \, 3}{{\rm Fix}\langle (12)(45) \rangle} & \\ \overset{Pattern \,6}{{\rm Fix}\langle (123)(456) \rangle} \ar[ru] & & } \end{eqnarray}

Without loss of generality, for all cases we assume that a critical point $x = (x_1, \ldots, x_6)$ satisfies $x_1 = 0$. \\

For $\Sigma=\langle (14) \rangle$,  we search for critical points of the form $(0,a,b,0,c,d)$ for $a, b, c, d \in {\mathbb R}$  mod $2\pi$: 
\begin{eqnarray}
2\sin(-a)+\sin(b-a)+\sin(d-a)=0  \nonumber \\ 
2\sin(-b)+\sin(a-b)+\sin(c-b)=0 \nonumber \\ 
2\sin(-c)+\sin(b-c)+\sin(d-c)=0 \nonumber \\ 
2\sin(-d)+\sin(a-d)+\sin(c-d)=0, \nonumber
\end{eqnarray}
which can be rewritten as  
\begin{eqnarray}
2\left[ \begin{array}{c} \sin a \\ \sin c \end{array}\right]= (\sin b+\sin d) \left[ \begin{array}{c} \cos a \\ \cos c \end{array}\right] -(\cos b +\cos d)\left[ \begin{array}{c} \sin a \\ \sin c \end{array}\right] \label{eq: G6 5} \\
2\left[ \begin{array}{c} \sin b \\ \sin d \end{array}\right]= (\sin a+\sin c)\left[ \begin{array}{c} \cos b \\ \cos d \end{array}\right] -(\cos a+\cos c)\left[ \begin{array}{c} \sin b \\ \sin d \end{array}\right] \label{eq: G6 6}
\end{eqnarray}
\noindent Case 1: If  $(\sin a, \sin c) , (\cos a, \cos c)$  are linearly independent,  then \eqref{eq: G6 5} implies that $\cos b+\cos d=-2$, and so $b=d=\pi$. Using that in \eqref{eq: G6 6}, we get  $\sin a= - \sin c$, and so $c=-a$ or $a=c+\pi$. These give the following two families of critical points:
$$(0,a,\pi,0,-a,\pi), \ \  (0,a,\pi,0,a+\pi,\pi).$$
If  $(\sin b, \sin d) , (\cos b, \cos d)$  are linearly independent, then analogously we get the other two families
$$(0,\pi,b,0,\pi,-b),  \ \  (0,\pi,b,0,\pi,b+\pi).$$

\noindent Case 2: If both  pairs of the previous case are linearly dependent, then each pair form a matrix with zero  determinant, and so  $\sin(c-a)=0$ and $\sin(d-b)=0$. Hence, $c=a$ or $c=a+\pi$ and $b=d$ or $b=d+\pi$. 
\begin{itemize} 
\item If $c=a$ and $d=b$, then  $\sin a=\sin(b-a)$ and $\sin b=\sin(a-b)$. So the critical points are $$(0,\pi,0,0,\pi,0), \ \ (0, \dfrac{4\pi}{3},\dfrac{2\pi}{3},0,\dfrac{4\pi}{3},\dfrac{2\pi}{3},0).$$
\item If $c=a$ and $d=b+ \pi$, then the critical points are $$(0,0,\pi,0,0,0), \ \ (0, \pi,b,0,\pi,b+\pi).$$
\end{itemize}

For $\Sigma=\langle (12)(45) \rangle$, we search for critical points of the form $(0,0,a,b,b,c)$, for $a, b, c \in {\mathbb R}$ mod $ 2\pi$: 
\begin{eqnarray}
2\sin(-a)+2\sin(b-a)=0  \nonumber\\
\sin(-b)+\sin(a-b)+\sin(c-b)=0 \nonumber \\
2\sin(-c)+2\sin(b-c)=0. \nonumber
\end{eqnarray}
These give the following two families of critical points:
$$ (0,0,a,\pi,\pi,-a), \ \ (0,0,a,\pi,\pi,a+\pi).$$

In \cite{AM} we have listed the eight  patterns of synchrony of $G_6$ that are distinct from the total synchrony pattern; see \cite[Table1]{AM}. Here we follow that order to present in Table~1 the compilation of all (families of) equilibrium points inside the corresponding synchrony pattern.  \\

We then apply Theorem~\ref{Bronskiteo} to the Jacobian matrix $Jf(x)$, at every critical point $x$, thought of as a weighted Laplacian of the graph. We refer to the beginning of Subsection~\ref{subsec:Laplacian maps} for the explanation.  

For the critical point $x = (0,0,{\pi}/{2},\pi,\pi,{3\pi}/{2})$, the Jacobian is 

$$ Jf(x) \ = \ \left( \begin{array}{rrrrrr} 0 & +1 & 0 &0 &-1 & 0 \\ +1 & 0 & 0 &-1 & 0 & 0 \\ 0 & 0 & 0 &0 &0 & 0  \\ 0 & -1 & 0 &0 &+1 & 0 \\ -1 & 0 & 0 & +1 &0 & 0 \\ 0 & 0 & 0 &0 &0 & 0 \end{array} \right).$$ 
Thought of as a graph Laplacian, the graph is as given in Fig.~\ref{fig:disconneted}.  We then have $c(G) = 3,$ $c(G_+) = 4,$  $c(G_{-}) = 4,$ which gives $1 \leq  n_+  \leq  2.$ This is one of the possible cases of pattern number 3 of Table~1. 

\begin{figure}[h]
	\centering 
	\includegraphics[width=4.2cm]{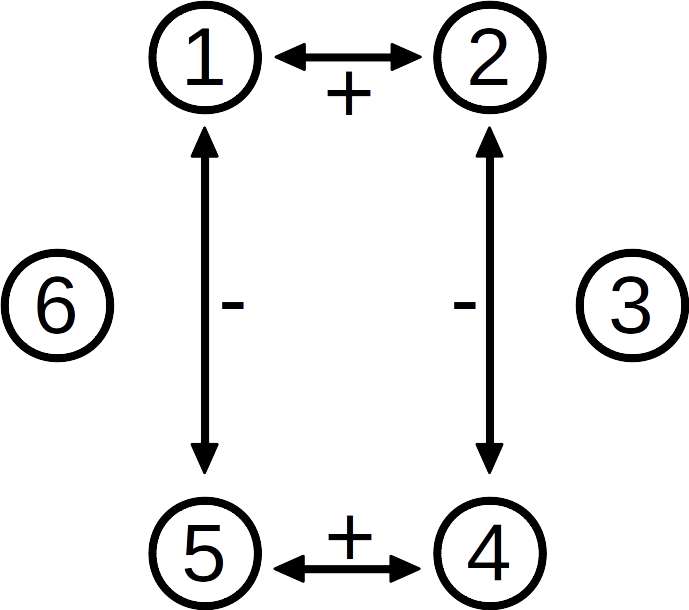} 
	
	\caption{The disconnected graph obtained from $G_6$ by assigning $Jf(x)$ as the weighted Laplacian, for $x = (0,0,{\pi}/{2},\pi,\pi,{3\pi}/{2})$. } \label{fig:disconneted}
\end{figure}
 It follows that for all cases in Table~1 we have $n_{+} \geq 1$. Therefore, all equilibria are unstable, and  asymptotic stability occurs only for the total synchrony subspace in the sense of Lyapunov as given in the previous subsection.
 
\begin{table}[p]
\centering 
\begin{tabular}{|c m{2.5cm}|m{4cm}|c|c|c|c|} 
\hline
 $\sharp$ & \ \ \ \  Pattern & \ \ \ \ \  Critical point &  $c(G_+)$ & $c(G_-)$ & $c(G)$ & $n_+({\cal L})$ 
 \\
 & of synchrony & \ \ \ \ \  representative & & & & \\ \hline 
 1 & {\includegraphics[scale=0.1]{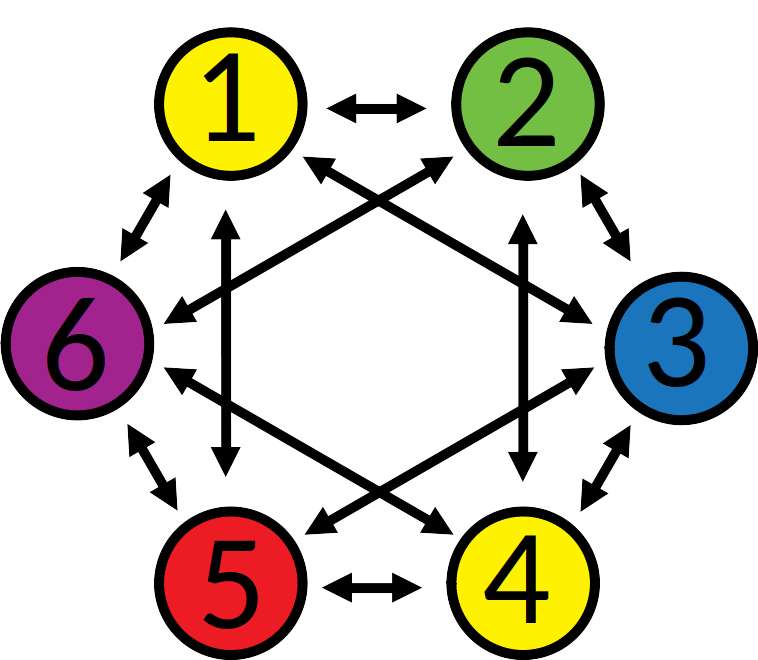}} 
&  \ \ \ \   \ \ \ \   \ \ \ \   \ \ -  &  - & - & - & -  \\
 \hline
 2 & {\includegraphics[scale=0.1]{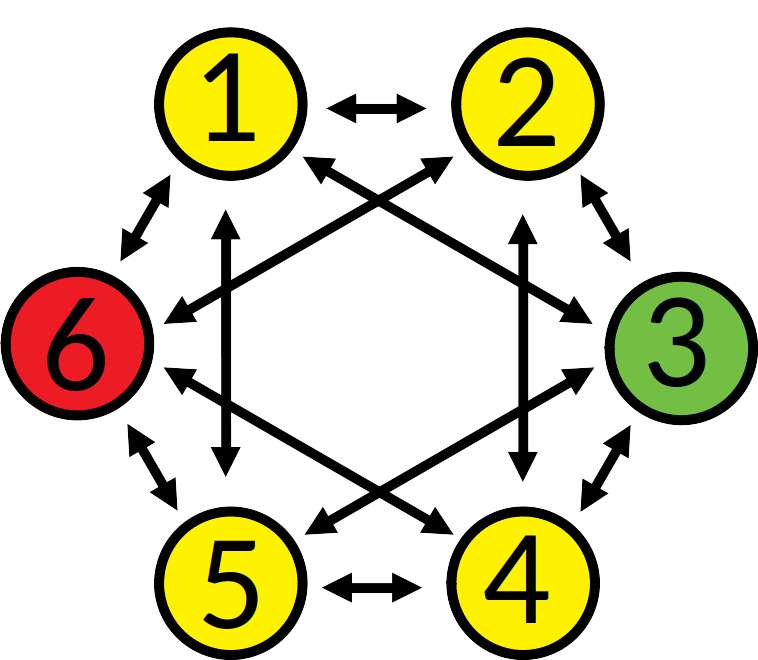}} &  $(0,0,\pi,0,0,0)$ &  2 & 2 & 1 & $1 \leq n_+ \leq 4$ 
  \\ \hline 3 & \multirow{3}{*} {\includegraphics[scale=0.1]{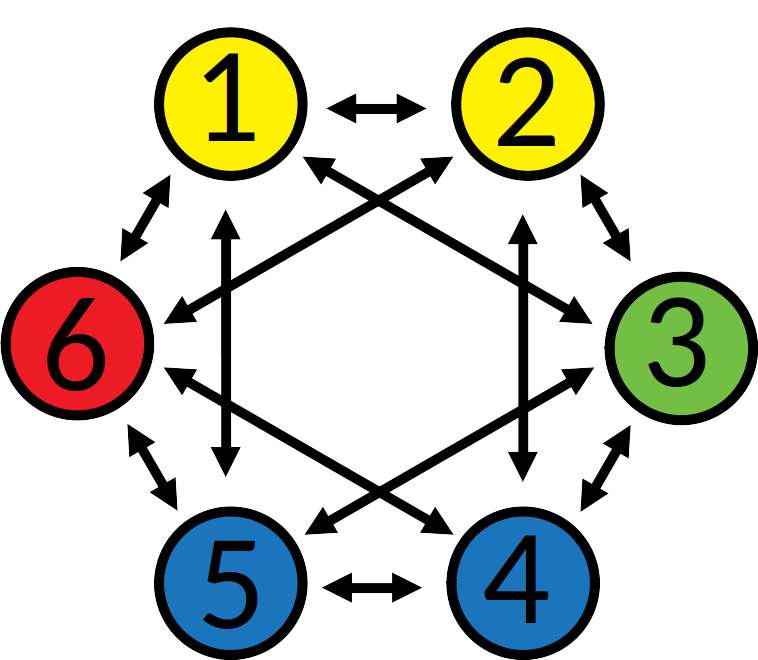}} &  $(0,0,\alpha, \pi,\pi,-\alpha)$   $(0,0,\alpha, \pi,\pi,\alpha+\pi)$  $\alpha \neq0,\pi, \pm \frac{\pi}{2}$&  2 & 1 & 1 & $1 \leq n_+ \leq 5 $ \\
 \cline{3-7}
 & &  $(0,0,\pm \frac{\pi}{2}, \pi,\pi,\mp \frac{\pi}{2})$  $(0,0,\pm \frac{\pi}{2}, \pi,\pi,\pm \frac{\pi}{2}+\pi)$ &  4 & 4 & 3 & $1 \leq n_+ \leq 2 $ \\&  & & & & & \\
 \hline
 4 & \multirow{3}{*}{\includegraphics[scale=0.1]{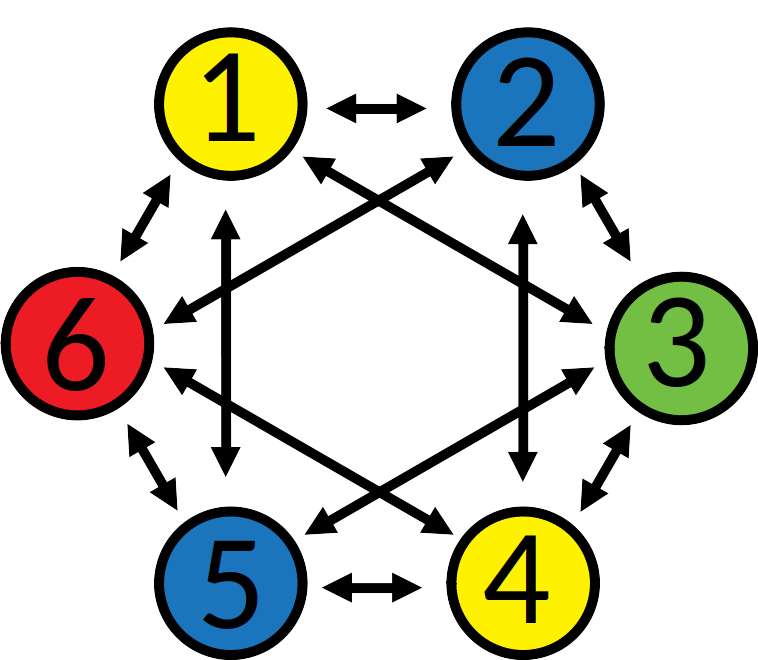}} &  $(0,\pi,\alpha,0 ,\pi,-\alpha)$ $(0,\pi,\alpha, 0,\pi,\alpha+\pi)$   $\alpha \neq0,\pi, \pm \frac{\pi}{2}$&  2 & 1 & 1 & $1 \leq n_+ \leq 5 $ \\
 \cline{3-7}
& &  $(0,\pi,\pm \frac{\pi}{2}, 0,\pi,\mp \frac{\pi}{2})$   $(0,\pi,\pm \frac{\pi}{2}, 0,\pi,\pm \frac{\pi}{2}+\pi)$ &  6 & 3 & 3 & $3 \leq n_+ \leq 3 $ \\
 \cline{3-7}
  & &  $(0,\pi,\pi,0,\pi,0)$ &  2 & 2 & 1 & $1 \leq n_+ \leq 4$ \\
 \hline
5 &  {\includegraphics[scale=0.1]{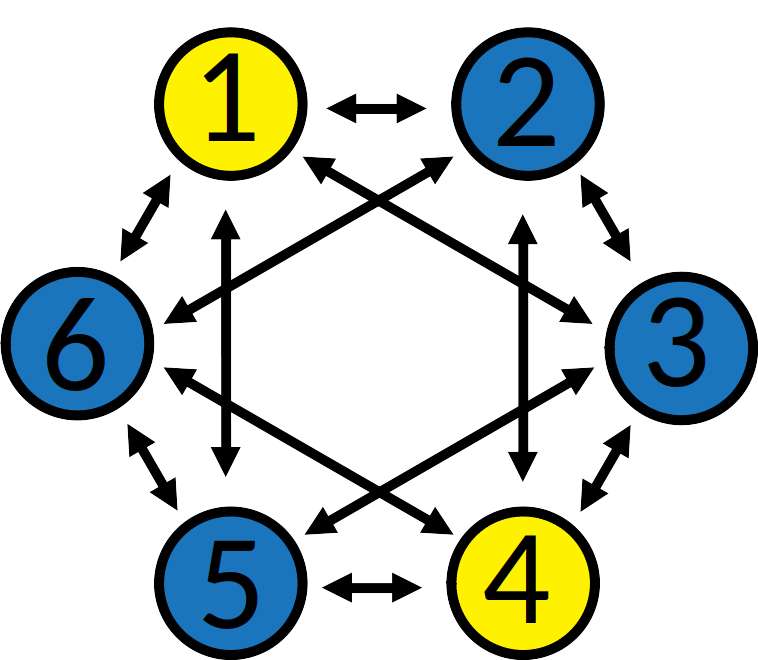}} &  $(0,\pi,0,0,\pi,0)$ &  3 & 1 & 1 & $2 \leq n_+ \leq 5$ \\
 \hline
 6 & {\includegraphics[scale=0.1]{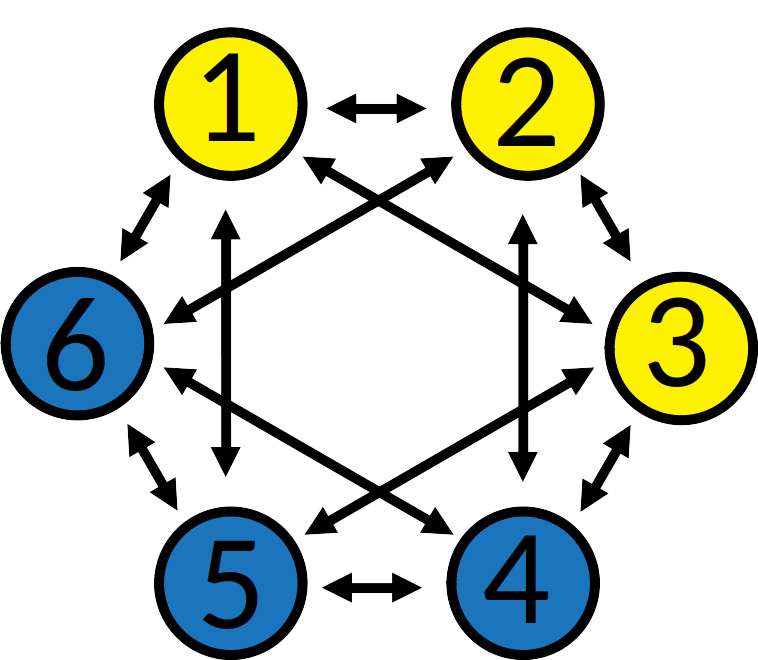}} &  $(0,0,0,\pi,\pi,\pi)$ &  2 & 1 & 1 & $1 \leq n_+ \leq 5$ \\
 \hline
 7 & {\includegraphics[scale=0.1]{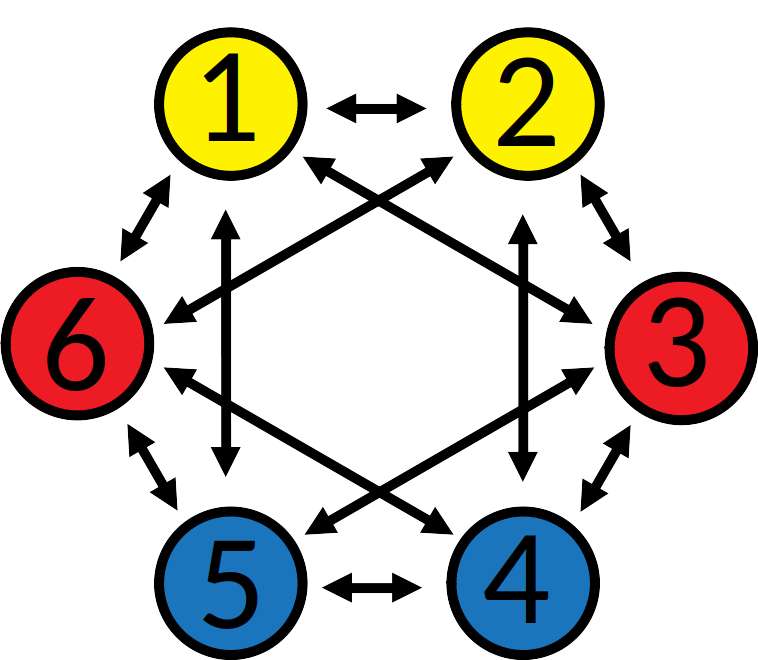}} &  $(0,0,0,\pi,\pi,0)$ &  2 & 1 & 1 & $1 \leq n_+ \leq 5$ \\
 \hline
 8 & {\includegraphics[scale=0.1]{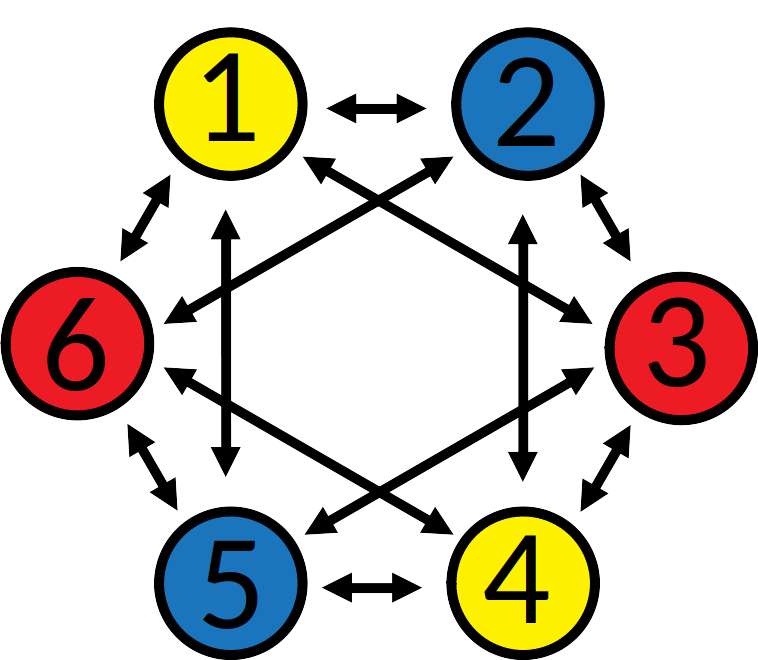}} &  $(0,\frac{4\pi}{3},\frac{2\pi}{3},0,\frac{4\pi}{3},\frac{2\pi}{3})$ &  6 & 1 & 1 & $5 \leq n_+ \leq 5$ \\
 \hline
\end{tabular}
\caption{The first column shows the synchrony patterns extracted from \cite[Table 1]{AM}; the corresponding class of critical points is given in the second column by a representative up to the symmetries of $Aut(G_6)$; the last column gives the number $n_{+}$ of positive eigenvalues of the Jacobian at the corresponding critical point; the remaining three columns are the data used to apply Theorem~\ref{Bronskiteo}. }
\end{table}

\subsection{A homogeneous network $\tilde{G}_6$ with two types of couplings} \label{subsec:two types of edges}

The aim of this example is to illustrate that there are stable equilibria from a `modification' of $G_6$. We consider a network system with six cells and coupling 
 defined from the graph $\tilde{G}_6$ given in Fig.~\ref{fig3}. The symmetry group of this graph is
 \[ Aut(\tilde{G}_6) \ = \ {\mathbb Z}_2^1 \times {\mathbb Z}_2^2 \times {\mathbb Z}_3,  \]
where  ${\mathbb Z}_2^1 = \langle (15)(24) \rangle$ and ${\mathbb Z}_2^2 = \langle(12)(36)(45) \rangle$, which are not conjugate, and 
${\mathbb Z}_3 = \langle (156)(234) \rangle.$ \\
 
 We choose the coupling functions  $\phi (\theta) =\sin \theta $,  which is represented by the continuous arrow type in Fig.~\ref{fig3}, and $\psi(\theta) = \theta$, represented  by the dashed arrow type, so the network system is  \\
\[ \begin{array}{ll} 
\dot{x}_1 = & \sin(x_2 - x_1) + \sin(x_3 - x_1) + x_5 + x_6 - 2x_1\\
\dot{x}_2 = & \sin(x_1 - x_2) + \sin(x_6 - x_2) + x_3 + x_4- 2x_2\\
\dot{x}_3 = & \sin(x_1 - x_3) + \sin(x_5 - x_3) + x_2 + x_4- 2x_3\\
\dot{x}_4 = & \sin(x_5 - x_4) + \sin(x_6 - x_4) + x_2 + x_3- 2x_4\\
\dot{x}_5 = & \sin(x_3 - x_5) + \sin(x_4 - x_5) + x_1 + x_6- 2x_5\\
\dot{x}_6 = & \sin(x_2 - x_6) + \sin(x_4 - x_6) + x_1 + x_5- 2x_6\\
\end{array} \]

\begin{figure}[h]

	\centering 
	\includegraphics[width=4.2cm]{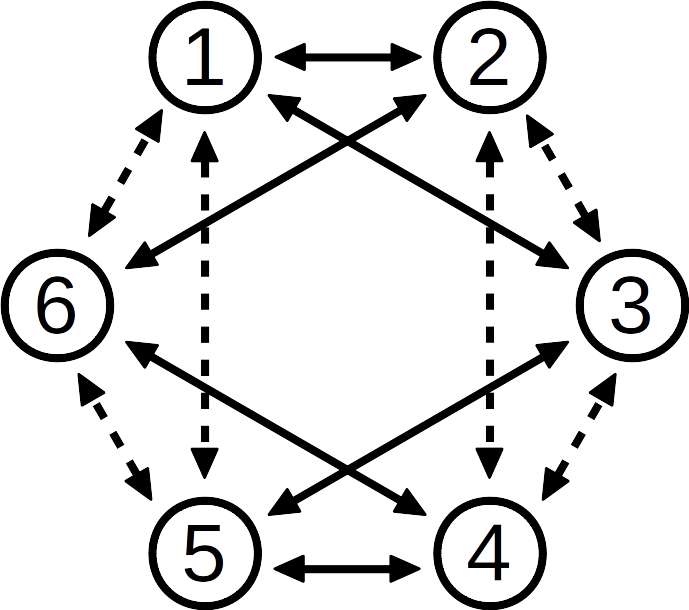} 
	\caption{A homogeneous $\tilde{G}_6$} \label{fig3}
\end{figure}
By direct investigation from the output of our code  (Remark~\ref{rmk: code}), there are no exotic synchrony patterns in this network graph. \\

The analysis follows analogously as in the previous subsection, so we shall just present a summary with the results. The diagram in (\ref{eq:diagram2}) is the lattice of non-trivial synchrony subspace up to conjugacy.

\begin{eqnarray} \label{eq:diagram2} {\xymatrix{ {\rm Fix}\langle(156)(234)   \rangle \ar[r] &  {\rm Fix}\langle(15)(24)   \rangle   \\ {\rm Fix}\langle(12)(36)(45),(15)(24)   \rangle \ar[r] \ar[ur] \ar[dr] & {\rm Fix}\langle (12)(36)(45)    \rangle \\ & {\rm Fix}\langle(14)(25)(36) \rangle }} 
\end{eqnarray}

For $\Sigma=\langle(15)(24)   \rangle$,  we search for critical points of the form $(0,a,b,a,0,c)$, for $a, b, c, \in {\mathbb R}$:  
\begin{eqnarray}
\sin(-a)+\sin(c-a)+(b-a)=0   \nonumber\\ 
2(a-b)+ 2 \sin(-b) = 0 \nonumber \\ 
-2c+\sin(a-c)=0, \nonumber
\end{eqnarray}
which imply that 
$$\sin b +b=a=c+\arcsin c,$$
Since $t \mapsto \sin(t)+t$  is injective, then $c=\sin b$. Together with the first equality, this gives
 $\sin a=-2 \sin b, $ and so  $\sin(b+\sin b) +2\sin b=0$.  Hence, $b=k\pi,k \in \mathbb{Z},$ and also  $a=b$.  
Therefore, the critical points are 
$$(0,k\pi,k\pi,k\pi,0,0), \quad \quad k \in \mathbb{Z}.$$

For $\Sigma=\langle (14)(25)(36) \rangle $,  we search for critical points of the form $(0,a,b,0,a,b)$, $a, b \in {\mathbb R}$: 
\begin{eqnarray}
\sin a+\sin b+a+b=0 \nonumber \\
\sin(b-a)+\sin(-a)-2b+a=0, \nonumber
\end{eqnarray}
which imply that $\sin a +a=\sin(-b)-b$, and then $a=-b$. Together with the second equality, this gives $\sin(2a)+2a=\sin(-a)-a$, and then $a=b=0$. This  gives the total synchrony, which has already been considered. \\

For $\Sigma=\langle (12)(36)(45)    \rangle $,  we search for critical points of the form $(a,a,0,b,b,0)$, $a, b \in {\mathbb R}$: \begin{eqnarray}
b=2a+\sin a \nonumber \\
a=2b+\sin b, \nonumber 
\end{eqnarray}
which also imply $a=b=0$.  
It follows from the computation of $n_{+}$ and $n_-$ by Theorem~\ref{Bronskiteo} that $(0,k\pi,k\pi,k\pi,0,0)$ is a one paramenter family of stable equilibria for $k$ even and unstable if $k$ is odd. \\

\noindent {\bf Acknowledgments.} TAA acknowledges financial support by FAPESP grant number 2019/2130-0. MM acknowledges financial support by FAPESP grant 2019/21181-0.

\bibliographystyle{alpha}

\begin{thebibliography}{100}


\bibitem{AguiarDiasAlg} Aguiar,  M.A.D. and Dias, A.P.S., The lattice of synchrony subspaces of a coupled cell network: Characterization and computation algorithm. {\it J. Non. Sci.} {\bf 24}(6)  949--996 (2014).
 
\bibitem{AguiarDias} Aguiar, M.A.D. and Dias, A.P.S., Minimal Coupled Cell Networks. \textit{Nonlinearity} \textbf{20} 193-219 (2007).

\bibitem{AM} Amorim, T.A., Manoel, M., The realization of admissible graphs for coupled vector fields (2022). Submitted. https://arxiv.org/abs/2212.07537.

\bibitem{AntoneliStewart} Antoneli, F. and Stewart, I., Symmetry and synchrony in coupled cell networks 1: Fixed-point spaces. {\it Int. J. Bif. Chaos} {\bf 16}(3)   559--577 (2006).

\bibitem{Bronski} Bronski, J.C.  and DeVille, L., Spectral theory for dynamics on graphs containing attractive and repulsive interactions. \textit{SIAM J. Appl. Math.} \textbf{74}(1)  83–105 (2014).

\bibitem{Bronski2} Bronski, J. C., Carty, T. E.,  DeVille, L., Synchronisation conditions in the Kuramoto model and their relationship to seminorms. \textit{Nonlinearity}, \textbf{34}(8), 5399  (2021).

\bibitem{CBV} Czajkowski, B. M., Batista, C. A.,  Viana, R. L., Synchronization of phase oscillators with chemical coupling: Removal of oscillators and external feedback control. \textit{Physica A: Statistical Mechanics and its Applications}, \textbf{610},  (2023).

\bibitem{DiasStewart}  Dias, A.P.S.  and Stewart, I., Linear equivalence and ODE-equivalence for coupled cell networks, \textit{Nonlinearity} \textbf{18} 1003–20 (2005).

\bibitem{JMB} Jadbabaie, A., Motee, N. Barahona, M., On the stability of the Kuramoto model of coupled nonlinear oscillators. \textit{Proceedings of the 2004 American Control Conference. IEEE},  4296--4301 (2004).


\bibitem{GSS} Golubitsky, M., Stewart, I.,  Schaeffer, D., {\it Singularities and Groups in Bifurcation Theory}, vol. 2, Appl. Math. Sci., Springer (1985).

\bibitem{SMT} Golubitsky, M., Stewart, I.,  T\"or\"ok, A., Patterns of synchrony in coupled cell networks with multiple arrows, \textit{SIAM J. Appl. Dyn. Syst.} \textbf{4}(1),
78–100 (2005).

\bibitem{livro} Golubitsky, M., Stewart, I., {\it Dynamics and Bifurcation in Networks: Theory and Applications of Coupled Differential Equations}, SIAM (2023).


\bibitem{SMM} Golubitsky, M., Nicol. M. and  Stewart, I., Some curious phenomena in coupled cell networks. \textit{Journal of Nonlinear Science} \textbf{14},  207-236 (2004).

\bibitem{KJA} Kotwal, T., Jiang, X., and Abrams, D. M., Connecting the Kuramoto model and the chimera state. \textit{Physical review letters} \textbf{119}(26), (2017).


\bibitem{Kuramoto1975}  Kuramoto, Y., Self-entrainment of a population of coupled non-linear oscillator, In International Symposium on Mathematical Problems in Theoretical Physics, \textit{Lecture Notes in Physics - Springer} \textbf{39}, New York, 420--422 (1975).


\bibitem{MR} Manoel, M., Roberts, M., Gradient systems on coupled cell networks, {\it Nonlinearity} {\bf 28}  3487--3509 (2015).

\bibitem{NB} Novikov, A.V., Benderskaya, E.N., Oscillatory neural networks based on the Kuramoto model for cluster analysis. \textit{Pattern Recognit. Image Anal}, \textbf{24}, 365–371 (2014).

\bibitem{Song} Song, J.U., Choi, K., Oh, S.M., Kahng, B., Exploring nonlinear dynamics and network structures in Kuramoto systems using machine learning approaches, {\it Chaos} {\bf 33} 073148 (2023).

\bibitem{Wiggins} Stephen, W., Introduction to applied nonlinear dynamical systems and chaos, Second Edition \textit{Springer-Verlag}  (2003). 

\bibitem{SMP} Stewart, I., Golubitsky, M., Pivato, M., Symmetry groupoids and patterns of synchrony in coupled cell networks, \textit{SIAM J. Appl. Dyn. Syst.}
\textbf{2}(4) 609–646 (2003).

\bibitem{VHMP} Vandermeer, J., Hajian-Forooshani, Z., Medina, N.,  Perfecto, I.,  New forms of structure in ecosystems revealed with the Kuramoto model. \textit{Royal Society open science}, {\bf 8}(3) (2021).

\end{thebibliography}

\end{document}